\documentclass[a4paper,UKenglish,cleveref, autoref, thm-restate]{lipics-v2021}

\title{An analytic calculus for intuitionistic belief} 

\author{Cosimo Perini Brogi}{University of Genoa, Italy \and \url{https://logicosimo.gitlab.io}}{perinibrogi@dima.unige.it}{https://orcid.org/0000-0001-7883-5727}{}

\authorrunning{C. Perini Brogi} 

\Copyright{Cosimo Perini Brogi} 

\nolinenumbers 

\hideLIPIcs  



\category{} 

\relatedversion{} 

\acknowledgements{}

\EventEditors{John Q. Open and Joan R. Access}
\EventNoEds{2}
\EventLongTitle{42nd Conference on Very Important Topics (CVIT 2016)}
\EventShortTitle{CVIT 2016}
\EventAcronym{CVIT}
\EventYear{2016}
\EventDate{December 24--27, 2016}
\EventLocation{Little Whinging, United Kingdom}
\EventLogo{}
\SeriesVolume{42}
\ArticleNo{23}

\usepackage{mathtools}
\usepackage{slashed}
\usepackage{picinpar}
\usepackage{float}
\usepackage{mathdots}
\usepackage{tikz-cd}
\usepackage{xypic}
\usepackage[all,2cell,cmtip]{xy} 
\UseAllTwocells 
\usepackage{relsize}
\usepackage{stmaryrd}
\usepackage{longtable}

\usepackage{bussproofs}
\usepackage{graphicx}
\usepackage{pstricks,pst-node,pst-tree, pstricks-add}

\begin{document}

\maketitle

\begin{abstract}
Intuitionistic belief has been axiomatized by Artemov and Protopopescu as an extension of intuitionistic propositional logic by means of the distributivity scheme K, and of co-reflection $A\rightarrow\Box A$. This way, belief is interpreted as a result of verification, and it fits an extended Brouwer-Heyting-Kolmogorov interpretation for intuitionistic propositional logic with an epistemic modality.
In the present paper, structural properties of a natural deduction system $\mathsf{IEL}^{-}$ for intuitionistic belief are investigated. The focus is on the analyticity of the calculus, so that the normalization theorem and the subformula property are proven firstly. From these, decidability and consistency of the logic follow as corollaries. Finally, disjunction properties, $\Box$-primality, and admissibility of reflection rule are established by using purely proof-theoretic methods.\footnote{This is a very rough draft that is intended as the second part of work-in-progress started with \cite{brogi}. For sure, many expository refinements are required to the present paper: it is basically a collection of rough results and reflections.}
\end{abstract}

\section*{Introduction}
Brouwer-Heyting-Kolmogorov (BHK) interpretation provides a semantics of mathematical statements in which the computational aspects of proving and refuting are highlighted.\footnote{See e.g.~\cite{troeslavan}.}
In this perspective, reasoning intuitionistically is similar to a safe mode of program execution which always terminates, and this analogy can be made precise by means of the so-called Curry-Howard correspondence between proof and programs of type theory \cite{sorensen}.

On the other hand, Kripke semantics for intuitionistic logic \cite{kripke} captures the informal idea of a process of growth of knowledge in time which characterises the mental life of the mathematician according to the founders of intuitionism.

It is worth-noting that the focuses of these semantics are quite different: BHK interpretation stresses the importance of the concept of proof in the semantics for intuitionistic logic; Kripke's approach highlights the epistemic process behind the provability of a statement.

Artemov and Protopopescu's \cite{artemov} suggests that the two views can safely coexist once a good computational interpretation of epistemic states is given. Their starting point is thus a BHK interpretation of epistemic statements in which knowledge and belief are considered as (different) results of a process of verification. The general idea is that a proof of a (mathematical) statement is a most strict type of verification, and that verifying a statement is a sufficient condition for \emph{believing} it. At the same time, \emph{knowing} that a statement is true means, according to this intuitionistic reading, that this very statement cannot be false, since we have a verification of it.

Hence, the proposed intuitionistic account of epistemic states validates a principle of ``constructivity of truth'' $$A\rightarrow \Box A$$ and of ``intuitionistic factivity of knowledge'' $$\Box A\rightarrow\neg\neg A.$$

The paper \cite{artemov} covers only axiomatic calculi and Kripke semantics for intuitionistic epistemic logics. However, besides the epistemic reading, a normal modality satisfying co-reflection -- i.e.~based on the same axiomatic calculus as discussed in~\cite{artemov} -- has shown relevant in programming by identifying the so-called applicative functors \cite{applicative}, which can be considered as intermediate objects between the Haskell type classes \texttt{Functor} and \texttt{Monad}. As a consequence, in \cite{litak} the logic of intuitionistic belief is thought of as a logic for applicative functors as well.

The present author has given a preliminary formal analysis of the computational reading of intuitionistic belief in \cite{brogi}, where a `minimalist' \textbf{natural deduction} system $\mathsf{IEL}^{-}$ for the intuitionistic logic of belief is developed and designed with the intent of translating it into a functional calculus for $\mathsf{IEL}^{-}$-deductions.

In the same work, it is given a very quick proof of the strong normalization theorem for that very natural deduction calculus, and it is shown that the belief modality can be interpreted also as a specific pointed monoidal endofunctor on the syntactic category of $\mathsf{IEL}^{-}$-proofs.

In the present paper, we \emph{prove something more}. We address the question of \textbf{analyticity} for $\mathsf{IEL}^{-}$, and in order to achieve a precise result, we will develop a \textbf{different proof} of \textbf{strong normalization}, from which we are able to derive the \textbf{subformula property} of normal $\mathsf{IEL}^{-}$-deductions.

By using this full normalization result, we give then \textbf{syntactic proofs} of several \textbf{proof-theoretic properties} for the system and investigate the structural behaviour of our natural deduction for intuitionistic belief.

In more details, we succeed in:
\begin{itemize}
\item Developing a proof of strong normalization for $\mathsf{IEL}^{-}$-deduction w.r.t.~both detour elimination and $\vee,\bot$-permutations;
\item Proving that $\mathsf{IEL}^{-}$ is analytical -- for enjoying the subformula property -- by the latter full normalization result;
\item Developing syntactic proofs of several proof-theoretic properties of $\mathsf{IEL}^{-}$, namely
\begin{itemize}
\item decidability;
\item canonicity of proofs;
\item consistency;
\item disjunction property;
\item admissibility of the reflection rule \AxiomC{$\Box A$}\UnaryInfC{$A$}\DisplayProof ;
\item modal disjunction property;
\item $\Box$-primality. 
\end{itemize} 
\end{itemize}  

The paper is then organised as follows: after recalling the natural deduction system $\mathsf{IEL}^{-}$ and its type theoretic counterpart for proof-terms in Section \ref{sec1}, we proceed with Section \ref{sec2} in proving that, in $\mathsf{IEL}^{-}$-deductions, detours can be eliminated, and that the permutations involving absurdity and disjunction can be safely converted, so that strong normalization holds for proof terms. In Section \ref{sec3}, we use that normalization result to derive the subformula property for normal deductions, and prove the structural properties of $\mathsf{IEL}^{-}$ previously mentioned; also, we recall the proof-theoretic semantics based on category theory for $\mathsf{IEL}^{-}$-deductions, to see that having permutations does not harm the semantic soundness w.r.t.~those structures. Finally, some remarks on future investigations and related works are made.

\section{System $\mathsf{IEL}^{-}$}\label{sec1}

In \cite{artemov}, intuitionistic belief is axiomatized as a calculus $\mathbb{IEL}^{-}$ given by:
\begin{itemize}
\item Axiom schemes for intuitionistic propositional logic;
\item Axiom scheme $\mathsf{K}: \Box(A\rightarrow B)\rightarrow\Box A\rightarrow\Box B$;
\item Axiom scheme of co-reflection: $A\rightarrow\Box A$;
\item Modus Pones \AxiomC{$A\rightarrow B$}\AxiomC{$A$}\RightLabel{$\mathsmaller{\mathit{MP}}$}\BinaryInfC{$B$}\DisplayProof as the only inference rule.
\end{itemize}
We write $\Gamma\vdash_{\mathbb{IEL}^{-}} A$ when $A$ is derivable in $\mathbb{IEL}^{-}$ assuming the set of hypotheses $\Gamma$, and we write $\mathbb{IEL}^{-}\vdash A$ when $\Gamma=\varnothing$.

These principles are chosen to axiomatise the idea that belief is the result of a verification in an intuitionistic framework where truth is given by provability, as expressed by the Brouwer-Heyting-Kolmogorov interpretation of intuitionistic logic.
Accordingly, co-reflection states that whatever is proven, it is also believed, since a proof is a most strict type of verification.
Within this extended BHK interpretation, the epistemic clause reads as follows~\cite{artemov}:
\begin{center}
a proof of $\Box A$ is conclusive evidence of verification that $A$ has a proof.
\end{center}

Notice that under this reading, the classical reflection scheme $\Box A\rightarrow A$ does not hold: $A$ can be verified without disclosing a specific proof. 

Besides its potential applications even outside mathematics, in~\cite{artemov} $\mathbb{IEL}^{-}$ is proven to be a normal modal logic which is sound and complete w.r.t.~a specific relational semantics. Moreover, its proper extension $\mathbb{IEL}$ is introduced to capture the state of knowledge by adding the scheme of intuitionistic factivity of truth $\Box A\rightarrow\neg\neg A$; as for belief, an adequate relational semantic is given, and some properties of the logics are discussed.

\subsection{System $\mathsf{IEL}^{-}$}

Since in that first work on intuitionistic belief the starting point for discussing epistemic states is a BHK interpretation of them -- therefore, an implicit computational semantics for belief and knowledge -- it seems natural to develop natural deduction systems that have straightforward type-theoretic counterparts: this is n important step towards making rigorous the computational interpretation of the modal operators introduced by the axioms systems and the relational structures of \cite{artemov}.

In the present paper, we focus on the the belief modality, and we introduce first the following natural deduction system:
\begin{definition}
Let $\mathsf{IEL}^{-}$ be the calculus extending the propositional fragment of $\mathsf{NJ}$ -- the natural deduction calculus for intuitionistic logic as presented in \cite{vandalen} -- by the following rule:
\begin{prooftree} \AxiomC{$\Gamma_{1}$}\noLine\UnaryInfC{$\vdots$}\noLine\UnaryInfC{$\Box A_{1}$}\AxiomC{$\,$}\noLine\UnaryInfC{$\,$}\noLine\UnaryInfC{$\cdots$}\AxiomC{$\Gamma_{n}$}\noLine\UnaryInfC{$\vdots$}\noLine\UnaryInfC{$\Box A_{n}$}\AxiomC{$[A_{1},\cdots , A_{n}],\Delta$}\noLine\UnaryInfC{$\vdots$}\noLine\UnaryInfC{$B$}\singleLine\RightLabel{$\mathsmaller{\Box-intro}$}\QuaternaryInfC{$\Box B$}
\end{prooftree}
where $\Gamma$ and $\Delta$ are \emph{sets of occurrences of formulae}, and all $A_{1},\cdots,A_{n}$ are discharged.
\end{definition}

This is the calculus introduced in~\cite{brogi}, and it differs from the system defined in \cite{depaivaeike} for $\mathbb{IK}$ by allowing the set $\Delta$ of additional hypotheses is the subdeduction of $B$.

As for the axiomatic system, we write $\Gamma\vdash_{\mathsf{IEL}^{-}} A$ when $A$ is derivable in $\mathsf{IEL}^{-}$ from the set of hypotheses $\Gamma$, and $\mathsf{IEL}^{-}\vdash A$ when $\Gamma=\varnothing$. 

It is straightforward to check that $\mathsf{IEL}^{-}$ is logically equivalent to $\mathbb{IEL}^{-}$:
\begin{proposition}$\Gamma\vdash_{\mathsf{IEL}^{-}}A$ iff $\Gamma\vdash_{\mathbb{IEL}^{-}}A$.
\end{proposition}
\begin{proof}
For both the directions we reason by induction on the derivation. See~\cite{brogi} for the details.
\end{proof}

\subsection{Modal $\lambda$-calculus}

Despite lacking a certain symmetry pertaining to traditional natural deduction calculi, deductions in the system $\mathsf{IEL}^{-}$ can be easily turned into derivation rules for $\lambda$-terms for a modal type theory involving the belief operator. In other terms, it is possible to extend the correspondence between intuitionistic natural deduction and typed $\lambda$-calculus -- namely, the Curry-Howard correspondence -- to cover $\mathsf{IEL}^{-}$ in a very natural way.

The correspondence for propositional operators can be summarised as follows
\begin{scriptsize}

\vspace{.28cm}
\begin{longtable}{l l l p{3.2cm}}
$f\equiv A$ & $\longmapsto$ & $x_{i}^{A}$, & where $i$ is the parcel of the hypothesis $A$ \\
 & & & \\
$f\equiv$\alwaysNoLine\AxiomC{$f_{1}$}\UnaryInfC{$A$}\AxiomC{$f_{2}$}\UnaryInfC{$B$}\alwaysSingleLine\BinaryInfC{$A\wedge B$}\DisplayProof &  $\longmapsto$ & $\langle t^{A},s^{B}\rangle$, & where $t^{A}$, $s^{B}$ correspond to $f_{1}$ and $f_{2}$ resp. \\
 & & & \\
$f\equiv$\alwaysNoLine\AxiomC{$f'$}\UnaryInfC{$A\wedge B$}\alwaysSingleLine\UnaryInfC{$A$}\DisplayProof &  $\longmapsto$ & $\pi_{1}.t^{A\times B}$, & where $t^{A\times B}$ corresponds to $f'$ \\
 & & & \\
$f\equiv$\alwaysNoLine\AxiomC{$f'$}\UnaryInfC{$A\wedge B$}\alwaysSingleLine\UnaryInfC{$B$}\DisplayProof &  $\longmapsto$ & $\pi_{2}.t^{A\times B}$, & where $t^{A\times B}$ corresponds to $f'$ \\
 & & & \\
$f\equiv$\alwaysNoLine\AxiomC{$f'$}\UnaryInfC{$B$}\alwaysSingleLine\UnaryInfC{$A\rightarrow B$}\DisplayProof &  $\longmapsto$ & $\lambda x_{i}^{A}. t^{B}$, & where $t^{B}$ corresponds to $f'$ and $i$ is the parcel of discharged hypotheses $A$ \\
 & & & \\
$f\equiv$\alwaysNoLine\AxiomC{$f_{1}$}\UnaryInfC{$A\rightarrow B$}\AxiomC{$f_{2}$}\UnaryInfC{$A$}\alwaysSingleLine\BinaryInfC{$B$}\DisplayProof &  $\longmapsto$ & $t^{A\rightarrow B}s^{A}$, & where $t^{A\rightarrow B}$, $s^{A}$ correspond to $f_{1}$ and $f_{2}$ resp. \\
& & & \\
$f\equiv$\alwaysNoLine\AxiomC{$f'$}\UnaryInfC{$A$}\alwaysSingleLine\UnaryInfC{$A\vee B$}\DisplayProof &  $\longmapsto$ & $\;\mathsf{in}_{1}.t^{A}$, & where $t^{A}$ corresponds to $f'$ \\
& & & \\
$f\equiv$\alwaysNoLine\AxiomC{$f'$}\UnaryInfC{$B$}\alwaysSingleLine\UnaryInfC{$A\vee B$}\DisplayProof &  $\longmapsto$ & $\;\mathsf{in}_{2}.s^{B}$, & where $s^{B}$ corresponds to $f'$ \\
& & & \\
$f\equiv$\AxiomC{$f'$}\noLine\UnaryInfC{$A\vee B$}\AxiomC{$[A]$}\noLine\UnaryInfC{$\vdots$}\noLine\UnaryInfC{$C$}\AxiomC{$[B]$}\noLine\UnaryInfC{$\vdots$}\noLine\UnaryInfC{$C$}\TrinaryInfC{$C$}\DisplayProof & $\longmapsto$ & $\mathsf{C}_{x^{A},y^{B}}(t,t_{1},t_{2})$ & where $\mathsf{C}$ bounds all occurrences of $x$ in $t_{1}$ and all occurrences of $y$ in $t_{2}$, and $t, t_{1}, t_{2}$ correspond to $f'$, the subdeduction of $C$ from $A$, and the subdeduction of $C$ from $B$, resp. \\
& & & \\
$f\equiv$\AxiomC{$f'$}\noLine\UnaryInfC{$\bot$}\UnaryInfC{$A$}\DisplayProof & $\longmapsto$ & $\mathsf{E}_{A}t$ & where $t$ corresponds to $f'$ \\
& & & \\
$f\equiv$\AxiomC{$f'$}\noLine\UnaryInfC{$A$}\UnaryInfC{$\top$}\DisplayProof & $\longmapsto$ & $(\mathsf{U}t)$ & where $t$ correspond to $f'$.
\end{longtable}
\end{scriptsize} 

The $\lambda$-term corresponding to $\Box$-introduction is then ruled by:
\begin{prooftree}
\AxiomC{$\Gamma_{1}\vdash t_{1}:\Box A_{1}$}\AxiomC{$\cdots$}\AxiomC{$\Gamma_{n}\vdash t_{n}:\Box A_{n}$}\AxiomC{$x_{1}:A_{1},\cdots,x_{n}:A_{n},\Delta\vdash s:B$}\QuaternaryInfC{$\Gamma_{1},\cdots,\Gamma_{n},\Delta\vdash\mathsf{B}_{x_{1},\cdots,x_{n}}(t_{1},\cdots,t_{n})\,\mathsf{in}\, s\;:\Box B$}
\end{prooftree}


\section{Normalization for $\mathsf{IEL}^{-}$-deductions}\label{sec2}

The modal $\lambda$-calculus just presented gives a neat notation for investigating the structural properties of $\mathsf{IEL}^{-}$ by allowing an equational reasoning on proof-terms.

As it is known, by imposing rewriting rules on $\mathsf{NJ}$-deductions -- as defined in~\cite{prawitz1971} --, we obtain the complete engine of \sloppy\mbox{$\lambda$-calculus} associated to that natural deduction.

For our modal $\lambda$-calculus, we need to add to system corresponding to the propositional fragment of $\mathsf{NJ}$ the following rewritings:
\begin{definition}[Modal rewritings]\label{riscritture}$\,$\\
\begin{enumerate}

\item $\mathsf{B}_{x_{1},\cdots,x_{i-1},x_{i},x_{i+1},\cdots,x_{n}}(t_{1},\cdots,t_{i-1},(\mathsf{B}_{\vec{y}}(\vec{s})\;\mathsf{in}\;t_{i}),t_{i+1},\cdots,t_{n})\;\mathsf{in}\; r\,\\ >_{M} \mathsf{B}_{x_{1},\cdots,x_{i-1},\vec{y},x_{i+1},\cdots,x_{n}}(t_{1},\cdots,t_{i-1},\vec{s},t_{i+1},\cdots,t_{n})\;\mathsf{in}\;r[x_{i}:=t_{i}]$
\item $\mathsf{B}_{x} t \;\mathsf{in}\; x >_{M} t$ 
\end{enumerate}
\end{definition}

In~\cite{brogi}, the present author already gave a quick proof of strong normalization and confluence for the full modal $\lambda$-calculus corresponding to $\mathsf{IEL}^{-}$ by tweaking a proof strategy due to~\cite{kakutani} for the implicational fragment of intuitionistic minimal modal logic $\mathbb{IK}$. Therefore, we already know that every $\mathsf{IEL}^{-}$-deductions has a unique normal form, as required.

That system of rewritings, however, does not suffice to establish the subformula property for $\mathsf{IEL}^{-}$: in order to achieve that, we need further reductions between proof-terms involving $\bot$- and $\vee$-elimination. Unfortunately, the translation adopted in that context does not preserve these additional rewritings, so that a different proof is required.

Our strategy consists of an extension of the method developed in \cite{degroote}, and, as in that paper, we use a modified CPS translation into simple type theory that is arithmetizable without using higher-order reasoning.

We commit the next pages to develop such a proof.

\subsection{Normalization for $\vee$-permutations}

Let $\lambda^{\Box\rightarrow\wedge\vee}$ denote the type theory corresponding to the fragment of $\mathsf{IEL}^{-}$ with $\Box$-modality, implication, conjunction and disjunctions as operators.

Its syntax is given by the following grammar:

$$ T ::= p\;|\; A \rightarrow B\; |\; A \times B\; |\; A + B\; |\; \Box A$$
\begin{center}

\begin{tabular}{l l}
$t ::=\;$ & $ x\; |\; \lambda x:A. t:B\; |\; t_{1}t_{2}\; |\; \langle t_{1},t_{2}\rangle\;|\;\pi_{1}(t)\;|\;\pi_{2}(t)\;|\;\mathsf{in}_{1}(a:A)\;|\;\mathsf{in}_{2}(b:B)\;|\;\mathsf{C}_{x,y}(t, t_{1},t_{2})\;|\;$\\ & $\mathsf{B}_{\vec{x}:\vec{A}}\,\vec{t}:\vec{\Box A} \;\mathsf{in}\; (s:B): \Box B$ .\\ 
\end{tabular}
\end{center}

As stated before, in order to obtain the subformula property for our natural deduction, the rewritings introduced in Definition~\ref{riscritture} are not enough. Therefore we add the following reductions:
\begin{definition}[Permutations for $\lambda^{\Box\rightarrow\wedge\vee}$]\label{permut}$\,$\\
\begin{enumerate}
\item $\mathsf{C}_{x,y}(t,t_{1},t_{2})s\,>_{P}\,\mathsf{C}_{x,y}(t,t_{1}s,t_{2}s)$
\item $\pi_{i}\mathsf{C}_{x,y}(t,t_{1},t_{2})\,>_{P}\,\mathsf{C}_{x,y}(t,\pi_{i}t_{1},\pi_{i}t_{2})$ for $i=1,2$
\item $C_{u,v}(C_{x,y}(t,t_{1},t_{2}),s_{1},s_{2})\,>_{P}\,\mathsf{C}_{x,y}(t,\mathsf{C}_{u,v}(t_{1},s_{1},s_{2}),\mathsf{C}_{u,v}(t_{2},s_{1},s_{2}))$
\item $\mathsf{B}_{x1,\cdots,x_{i-1},x_{i},x_{i+1},\cdots,x_{n}}\,(t_{1},\cdots,t_{i-1},\mathsf{C}_{x,y}(t,s_{1},s_{2}),t_{i+1},\cdots,t_{n})\,\mathsf{in}\,s\,>_{P}\\
\mathsf{C}_{x,y}(t,\mathsf{B}_{x1,\cdots,x_{i-1},x_{i},x_{i+1},\cdots,x_{n}}(t_{1},\cdots,t_{i-1},s_{1},t_{i+1},\cdots,t_{n})\,\mathsf{in}\,s,\\\;\mathsf{B}_{x1,\cdots,x_{i-1},x_{i},x_{i+1},\cdots,x_{n}}(t_{1},\cdots,t_{i-1},s_{2},t_{i+1},\cdots,t_{n})\,\mathsf{in}\,s)$.
\end{enumerate}
\end{definition}

We now assign a norm to $\lambda^{\Box\rightarrow\wedge\vee}$-terms.

\begin{definition}[Permutation degree for $\lambda^{\Box\rightarrow\wedge\vee}$]$\,$\\
\begin{enumerate}
\item $|x|=1$
\item $|\lambda x.t|=|t|$
\item $|ts|=|t|+\#t\times|s|$
\item $|\langle t,s\rangle|=|t|+|s|$
\item $|\pi_{i}t|=|t|+\#t$ for $i=1,2$
\item $|\mathsf{in}_{i}t|=|t|$ for $i=1,2$
\item $|\mathsf{C}_{x,y}(t,t_{1},t_{2})|=|t|+\#t\times(|t_{1}|+|t_{2}|)$
\item $|\mathsf{B}_{x_{1},\cdots,x_{n}}\,(t_{1},\cdots,t_{n})\,\mathsf{in}\, s|=|s|\times\prod^{n}_{i=0}|t_{i}|+\prod_{i=0}^{n}\#t_{i}$

where

\item $\#x=1$
\item $\#\lambda x.t=1$
\item $\#ts=\#t$
\item $\#\langle t,s\rangle=1$
\item $\#\pi_{i}t=\#t$ for $i=1,2$
\item $\#\mathsf{in}_{i}t=1$ for $i=1,2$
\item $\#\mathsf{C}_{x,y}(t,t_{1},t_{2})=2\times\#t\times(\#t_{1}+\#t_{2})$
\item $\#\mathsf{B}_{x_{1},\cdots,x_{n}}\,(t_{1},\cdots,t_{n})\,\mathsf{in}\, s=\#s\times\prod_{i=0}^{n}\#t_{i}$
\end{enumerate}

\end{definition}

Now we prove that this norm decreases after a permutation.

\begin{lemma}\label{lem6}For any $\lambda^{\Box\rightarrow\wedge\vee}$-terms $t,s$, if $t>_{P}s$, then $\#t=\#s$.
\end{lemma}
\begin{proof}
Cases involving traditional operators are dealt with in~\cite[lemma 4]{degroote}. 

We just have to prove the claim for the $\Box,\vee$-permutation:

\vspace{.63cm}

\begin{tabular}{l l l}
$\#\mathsf{B}_{x1,\cdots,x_{i-1},x_{i},x_{i+1},\cdots,x_{n}}\,(t_{1},\cdots,t_{i-1},\mathsf{C}_{x,y}(t,s_{1},s_{2}),t_{i+1},\cdots,t_{n})\,\mathsf{in}\,s$ & $=$ & \\
$\#s\times\#\mathsf{C}_{x,y}(t,s_{1},s_{2})\times\prod_{j=0}^{i-1}\#t_{j}\times\prod_{j=i+1}^{n}\#t_{j}$ & $=$ & \\
 $\#s\times 2\times\#t\times(\#s_{1}+\#s_{2})\times\prod_{j=0}^{i-1}\#t_{j}\times\prod_{j=i+1}^{n}\#t_{j}$ . \\ 
\end{tabular}

\vspace{.63cm}

\begin{tabular}{l l l}
$\#\mathsf{C}_{x,y}(t,\mathsf{B}_{x1,\cdots,x_{i-1},x_{i},x_{i+1},\cdots,x_{n}}(t_{1},\cdots,t_{i-1},s_{1},t_{i+1},\cdots,t_{n})\,\mathsf{in}\,s,$ \\ $\qquad\,\,\,\,\mathsf{B}_{x1,\cdots,x_{i-1},x_{i},x_{i+1},\cdots,x_{n}}(t_{1},\cdots,t_{i-1},s_{2},t_{i+1},\cdots,t_{n})\,\mathsf{in}\,s)$ & $=$ & \\
$2\times\#t\times(\#\mathsf{B}_{x1,\cdots,x_{i-1},x_{i},x_{i+1},\cdots,x_{n}}(t_{1},\cdots,t_{i-1},s_{1},t_{i+1},\cdots,t_{n})\,\mathsf{in}\,s\;+$ \\
$\qquad\qquad\,\#\mathsf{B}_{x1,\cdots,x_{i-1},x_{i},x_{i+1},\cdots,x_{n}}(t_{1},\cdots,t_{i-1},s_{2},t_{i+1},\cdots,t_{n})\,\mathsf{in}\,s)$ & $=$ \\
$2\times\#t\times(\#s\times\prod_{j=0}^{i-1}\#t_{j}\times\#s_{1}\times\prod_{j=i+1}^{n}\#t_{j}+\#s\times\prod_{j=0}^{i-1}\#t_{j}\times\#s_{2}\times\prod_{j=i+1}^{n}\#t_{j})$ & $=$ \\
$\#s\times 2\times\#t\times(\#s_{1}+\#s_{2})\times\prod_{j=0}^{i-1}\#t_{j}\times\prod_{j=i+1}^{n}\#t_{j}$ .
\end{tabular}
\end{proof}

\begin{lemma}For any $\lambda^{\Box\rightarrow\wedge\vee}$-terms $t,s$, if $t>_{P}s$, then $|t|>|s|$.
\end{lemma}

\begin{proof}
As for Lemma~\ref{lem6}, cases involving traditional operators are dealt with in~\cite[lemma 5]{degroote}. It remains to prove the claim for the $\Box,\vee$-permutation only:

\begin{tabular}{l l l}
$|\mathsf{B}_{x1,\cdots,x_{i-1},x_{i},x_{i+1},\cdots,x_{n}}\,(t_{1},\cdots,t_{i-1},\mathsf{C}_{x,y}(t,s_{1},s_{2}),t_{i+1},\cdots,t_{n})\,\mathsf{in}\,s|$ & $=$ & \\
 $|s|\times|\mathsf{C}_{x,y}(t,s_{1},s_{2})|\times\prod_{j=0}^{i-1}|t_{j}|\times\prod_{j=i+1}^{n}|t_{j}|+\prod_{j=0}^{i-1}\#t_{j}\times\#\mathsf{C}_{x,y}(t,s_{1},s_{2})\times\prod_{j=i+1}^{n}\#t_{j}$ & $=$ \\
 $(|s|\times\prod_{j=0}^{i-1}|t_{j}|\times\prod_{j=i+1}^{n}|t_{j}|\times(|t|+\#t\times(|s_{1}|+|s_{2}|))\;+$ \\
 $\prod_{j=0}^{i-1}\#t_{j}\times\prod_{j=i+1}^{n}\#t_{j}\times(2\times\#t\times(\#s_{1}+\#s_{2})))$ & $=$ \\
  $(|s|\times\prod_{j=0}^{i-1}|t_{j}|\times\prod_{j=i+1}^{n}|t_{j}|\times|t|\;+$ \\
  $\#t\times|s|\times\prod_{j=0}^{i-1}|t_{j}|\times\prod_{j=i+1}^{n}|t_{j}|\times|s_{1}|+\#t\times|s|\times\prod_{j=0}^{i-1}|t_{j}|\times\prod_{j=i+1}^{n}|t_{j}|\times|s_{2}|\,+$ \\
  $2\times\#t\times\prod_{j=0}^{i-1}\#t_{j}\times\prod_{j=i+1}^{n}\#t_{j}\times\#s_{1}+2\times\#t\times\prod_{j=0}^{i-1}\#t_{j}\times\prod_{j=i+1}^{n}\#t_{j}\times\#s_{1})$ & $>$ \\
  $(|t|+\#t\times|s|\times\prod_{j=0}^{i-1}|t_{j}|\times\prod_{j=i+1}^{n}|t_{j}|\times|s_{1}|+\#t\times|s|\times\prod_{j=0}^{i-1}|t_{j}|\times\prod_{j=i+1}^{n}|t_{j}|\times|s_{2}|\;+$ \\
  $\#t\times\prod_{j=0}^{i-1}\#t_{j}\times\prod_{j=i+1}^{n}\#t_{j}\times\#s_{1}+\#t\times\prod_{j=0}^{i-1}\#t_{j}\times\prod_{j=i+1}^{n}\#t_{j}\times\#s_{2})$ & $=$ \\
  $(|t|+\#t\times(|s|\times\prod_{j=0}^{i-1}|t_{j}|\times\prod_{j=i+1}^{n}|t_{j}|\times|s_{1}|+|s|\times\prod_{j=0}^{i-1}|t_{j}|\times\prod_{j=i+1}^{n}|t_{j}|\times|s_{2}|\;+$ \\
  $\prod_{j=0}^{i-1}\#t_{j}\times\prod_{j=i+1}^{n}\#t_{j}\times\#s_{1}+\times\prod_{j=0}^{i-1}\#t_{j}\times\prod_{j=i+1}^{n}\#t_{j}\times\#s_{2}))$ & $=$ \\
  $(|t|+\#t\times(|\mathsf{B}_{x1,\cdots,x_{i-1},x_{i},x_{i+1},\cdots,x_{n}}(t_{1},\cdots,t_{i-1},s_{1},t_{i+1},\cdots,t_{n})\,\mathsf{in}\,s|\;+$ \\
  $|\mathsf{B}_{x1,\cdots,x_{i-1},x_{i},x_{i+1},\cdots,x_{n}}(t_{1},\cdots,t_{i-1},s_{2},t_{i+1},\cdots,t_{n})\,\mathsf{in}\,s|))$ & $=$\\
  $|(\mathsf{C}_{x,y}(t,\mathsf{B}_{x1,\cdots,x_{i-1},x_{i},x_{i+1},\cdots,x_{n}}(t_{1},\cdots,t_{i-1},s_{1},t_{i+1},\cdots,t_{n})\,\mathsf{in}\,s,$ \\ $\qquad\,\,\,\,\mathsf{B}_{x1,\cdots,x_{i-1},x_{i},x_{i+1},\cdots,x_{n}}(t_{1},\cdots,t_{i-1},s_{2},t_{i+1},\cdots,t_{n})\,\mathsf{in}\,s)|$,
\end{tabular}
as required.
\end{proof}

As an immediate consequence we have strong normalization w.r.t.~$\vee$-permutations:

\begin{lemma}\label{normperm}$\lambda^{\Box\rightarrow\wedge\vee}$-calculus is strongly normalizing w.r.t.~$>_{P}$.
\end{lemma}

Notice that this result does not involve types, but it, as previously stated, is mandatory for obtaining the subformula property for the typed system.

\subsection{Detour elimination}
We want to eliminate useless steps of computation from our $\mathsf{IEL}^{-}$-deductions. For the moment, we shall restrict to the $\lambda^{\Box\rightarrow\wedge\vee}$-calculus for proof-terms, and extend the standard rewritings for the corresponding fragment of intuitionistic natural deduction by the reduction $>_{M}$ introduced in Definition~\ref{riscritture}. The resulting system is the following:
\begin{definition}[Detour conversions for $\lambda^{\Box\rightarrow\wedge\vee}$]$\,$\\
\begin{enumerate}
\item $(\lambda x.t)s\,>_{D}\,t[x:=s]$
\item $\pi_{i}\langle t_{1},t_{2}\rangle\,>_{D}\,t_{i}$ for $i=1,2$
\item $C_{x_{1},x_{2}}(\mathsf{in}_{i}t,t_{1},t_{2})\,>_{P}\,t_{i}[x_{i}:=t]$ for $i=1,2$
\item $\mathsf{B}_{x_{1},\cdots,x_{i-1},x_{i},x_{i+1},\cdots,x_{n}}(t_{1},\cdots,t_{i-1},(\mathsf{B}_{\vec{y}}(\vec{s})\;\mathsf{in}\;t_{i}),t_{i+1},\cdots,t_{n})\;\mathsf{in}\; r\,\\ >_{D} \mathsf{B}_{x_{1},\cdots,x_{i-1},\vec{y},x_{i+1},\cdots,x_{n}}(t_{1},\cdots,t_{i-1},\vec{s},t_{i+1},\cdots,t_{n})\;\mathsf{in}\;r[x_{i}:=t_{i}]$
\item $\mathsf{B}_{x} t \;\mathsf{in}\; x >_{D} t$ 

\end{enumerate}
\end{definition}

In general, $>_{\ldots}$ will denote a one-step reduction relation between $\lambda$-terms. In the following, we use $\overset{\mathsmaller{+}}{>_{\ldots}}$ for the transitive closure of $>_{\ldots}$, and $\gg_{\ldots}$ for its reflexive transitive closure.

As stated before, it is possible to show that $\mathsf{IEL}^{-}$-deductions do normalise w.r.t.~$>_{D}$ by simulating modal rewritings as $\rightarrow$-rewritings in simple type theory~\cite{brogi}. But since we want to consider permutations also, that strategy must be enforced to make the underlying translation compatible with $>_{P}$.

In the following, we will see that it is possible to reduce $\lambda^{\Box\rightarrow\wedge\vee}$ to simple type theory -- where $\rightarrow$ is the only type -- by a modified general negative translation that is able to map normal forms into normal forms.
Again, we accommodate the definitions in~\cite[§4]{degroote} to cover the modal constructions.

\begin{definition}[CPS-translation] Let $\overline{A}$ be the translation for the type $A$ of $\lambda^{\Box\rightarrow\wedge\vee}$ defined by:
$$\overline{A}\,=\,\sim\sim A^{\circ}$$ where $\sim A = A\rightarrow q$ for a specific atomic type $q$, and where
\begin{itemize}
\item $p^{\circ}=p$
\item $(A\rightarrow B)^{\circ}=\overline{A}\rightarrow\overline{B}$
\item $(A\wedge B)^{\circ}=\,\sim(\overline{A}\rightarrow\sim\overline{B})$
\item $(A\vee B)^{\circ}=\,\sim\overline{A}\rightarrow\sim\sim\overline{B}$
\item $(\Box A)^{\circ}=\,\sim\sim\overline{A}$
\end{itemize}

For terms, we define
\begin{itemize}
\item $\overline{x}=\lambda k. xk$
\item $\overline{\lambda x.t}=\lambda k.k(\lambda x.\overline{t})$
\item $\overline{ts}=\lambda k.\overline{t}(\lambda m. m\overline{s}k)$
\item $\overline{\langle t,s\rangle}=\lambda k.k(\lambda u. u\overline{t}\overline{s})$
\item $\overline{\pi_{1}t}=\lambda k.\overline{t}(\lambda u.u(\lambda i.\lambda j.ik))$
\item $\overline{\pi_{2}t}=\lambda k.\overline{t}(\lambda u.u(\lambda i.\lambda j.jk))$
\item $\overline{\mathsf{in}_{1}t}=\lambda k.k(\lambda i.\lambda j.i\overline{t})$
\item $\overline{\mathsf{in}_{2}t}=\lambda k.k(\lambda i.\lambda j.j\overline{t})$
\item $\overline{\mathsf{C}_{x,y}(t,t_{1},t_{2})}=\lambda k.\overline{t}(\lambda m.m(\lambda x.\overline{t_{1}}k)(\lambda y.\overline{t_{2}}k))$
\item $\overline{\mathsf{B}_{x_{1},\cdots,x_{n}}(t_{1},\cdots,t_{n})\,\mathsf{in}\,s}=\lambda k.\overline{t_{1}}(\lambda x_{1}.\cdots\overline{t_{n}}(\lambda x_{n}.k\overline{s})\cdots)$
\end{itemize} 

where $k,m,u,i,j$ are fresh variables.
\end{definition}

It is now easy to prove that this translation preserves typing:

\begin{proposition}\label{typing} Let $t$ be a $\lambda^{\Box\rightarrow\wedge\vee}$-term with type $A$ in context $\Gamma$. Then $\overline{t}$ is a term of simple type theory with type $\overline{A}$ in context $\overline{\Gamma}$.
\end{proposition}
\begin{proof}
By induction on the structure of $t$. We consider the case $t=\mathsf{B}_{x_{1},\cdots,x_{n}}(t_{1},\cdots,t_{n})\,\mathsf{in}\,s$:
\begin{prooftree}
\AxiomC{$\overline{\Gamma_{1}}\vdash \overline{t_{1}}:\,\sim\sim\overline{A_{1}}$}
\AxiomC{$\overline{\Gamma_{n}}\vdash \overline{t_{n}}:\,\sim\sim\overline{A_{n}}$}\AxiomC{$x_{1}:\overline{A_{1}},\cdots,x_{n}:\overline{A_{n}},\overline{\Delta}\vdash \overline{s}:\overline{B}$}\AxiomC{$k:\,\sim\overline{B}\vdash k:\,\sim\overline{B}$}\BinaryInfC{$x_{1}:\overline{A_{1}},\cdots,x_{n}:\overline{A_{n}},\overline{\Delta},k:\,\sim\overline{B}\vdash k\overline{s}:q$}\UnaryInfC{$x_{1}:\overline{A_{1}},\cdots,x_{n-1}:\overline{A_{n-1}},\overline{\Delta},k:\,\sim\overline{B}\vdash \lambda x_{n}.k\overline{s}:\,\sim\overline{A_{n}}$}\BinaryInfC{$\overline{\Gamma_{n}},x_{1}:\overline{A_{1}},\cdots,x_{n-1}:\overline{A_{n-1}},\overline{\Delta},k:\,\sim\overline{B}\vdash t_{n}(\lambda x_{n}.k\overline{s}):q$}\UnaryInfC{$\overline{\Gamma_{n}},x_{1}:\overline{A_{1}},\cdots,x_{n-2}:\overline{A_{n-2}},\overline{\Delta},k:\,\sim\overline{B}\vdash \lambda x_{n-1}. t_{n}(\lambda x_{n}.k\overline{s}):\,\sim\overline{A_{n-1}}$}\noLine\UnaryInfC{$\vdots$}\noLine\UnaryInfC{$\overline{\Gamma_{2}},\cdots\overline{\Gamma_{n}},\overline{\Delta},k:\,\sim\overline{B}\vdash (\lambda x_{1}.\cdots\overline{t_{n}}(\lambda x_{n}.k\overline{s})\cdots):\,\sim\overline{A_{1}}$}\BinaryInfC{$\overline{\Gamma_{1}},\cdots\overline{\Gamma_{n}},\overline{\Delta},k:\,\sim\overline{B}\vdash \overline{t_{1}}(\lambda x_{1}.\cdots\overline{t_{n}}(\lambda x_{n}.k\overline{s})\cdots):q$}\UnaryInfC{$\overline{\Gamma_{1}},\cdots\overline{\Gamma_{n}},\overline{\Delta}\vdash \lambda k.\overline{t_{1}}(\lambda x_{1}.\cdots\overline{t_{n}}(\lambda x_{n}.k\overline{s})\cdots):\,\sim\sim\overline{B}$}
\end{prooftree}
\end{proof}

However, for terms we need some further accommodations in order to avoid redexes:

\begin{definition}[Modified CPS-translation] The modified CPS-translation $\overline{\overline{t}}$ for any $\lambda^{\Box\rightarrow\wedge\vee}$-term is $$\overline{\overline{t}}\,=\,\lambda k.(M : k),$$
where $k$ is a fresh variable, and where the infix operator $:$ is defined by induction on the structure of $t$:
\begin{itemize}
\item $x:r=xr$
\item $\lambda x. t : r = r(\lambda x.\overline{\overline{t}})$
\item $ts : r = t : \lambda m. m\overline{\overline{s}}r$
\item $\langle t,s\rangle : r = r(\lambda u. u\overline{\overline{t}}\overline{\overline{s}})$
\item $\pi_{1}t : r = t : \lambda u. u(\lambda i.\lambda j. i r)$
\item $\pi_{2}t : r = t : \lambda u. u(\lambda i.\lambda j. j r)$
\item $\mathsf{in}_{1}t : r = r(\lambda i.\lambda j. i\overline{\overline{t}})$
\item $\mathsf{in}_{2}t : r = r(\lambda i.\lambda j. j\overline{\overline{t}})$
\item $\mathsf{C}_{x,y}(t,t_{1},t_{2}) : r = t : \lambda m.(\lambda x.(t_{1} : r))(\lambda y.(t_{2} : r))$
\item $\mathsf{B}_{x_{1},\cdots,x_{n}}(t_{1},\cdots,t_{n})\,\mathsf{in}\,s : r = t_{1} : (\lambda x_{1}.\cdots t_{n} : (\lambda x_{n}. r \overline{\overline{s}})\cdots)$
\end{itemize}

where $m,u,i,j$ are fresh variables and $x,y,x_{1},\cdots,x_{n}$ do not occur free in $r$.
\end{definition}


\begin{lemma}\label{redex} Let $t,r$ be $\lambda^{\Box\rightarrow\wedge\vee}$-terms. Then
\begin{enumerate}
\item $\overline{t}\gg_{\beta\eta}\overline{\overline{t}}$
\item $\overline{t}r\gg_{\beta\eta} t : r$
\end{enumerate}
\end{lemma}
\begin{proof}
By simultaneous induction for property 1 and property 2 on the structure of $t$.
\end{proof}

Moreover, typing is respected by the modified translation, as expected.
\begin{proposition}\label{typing2} Let $t$ be a $\lambda^{\Box\rightarrow\wedge\vee}$-term with type $A$ in context $\Gamma$. Then $\overline{\overline{t}}$ is a term of simple type theory with type $\overline{A}$ in context $\overline{\Gamma}$.
\end{proposition}
\begin{proof}
By proposition \ref{typing}, and lemma \ref{redex}.
\end{proof}

The next lemmas show that we can simulate detours by $\beta\eta$-reductions.

\begin{lemma}\label{lem13} For $\lambda^{\Box\rightarrow\wedge\vee}$-terms $t,s$ and any term $r$ of simple type theory that has no free occurrences of $x$, the following hold:
\begin{enumerate}
\item $(t : r)[x:=\overline{\overline{s}}]\gg_{\beta\eta}\,(t[x:=s]) : r$
\item $\overline{\overline{t}}[x:=\overline{\overline{s}}]\gg_{\beta\eta}\overline{\overline{t[x:=s]}}.$
\end{enumerate}
\end{lemma}
\begin{proof}
Property 1 is proven by induction on the structure of $t$. Property 2 then follows.
\end{proof}

\begin{lemma} For any terms $r,s$ of simple type theory such that $r\overset{\mathsmaller{+}}{>}_{\beta\eta}s$, we have $t : r\,\overset{\mathsmaller{+}}{>}_{\beta\eta}\, t : s$ for any $\lambda^{\Box\rightarrow\wedge\vee}$-term $t$.
\end{lemma}
\begin{proof}
By straightforward induction on the structure of $t$.
\end{proof}

\begin{lemma}\label{maind} Let $t,s$ be two $\lambda^{\Box\rightarrow\wedge\vee}$-terms such that $t>_{D}s$. Then:
\begin{enumerate}
\item $t : r\overset{\mathsmaller{+}}{>}_{\beta\eta} s : r$ for any term $r$ of simple type theory,
\item $\overline{\overline{t}}\overset{\mathsmaller{+}}{>}_{\beta\eta}\overline{\overline{s}}$.
\end{enumerate}
\end{lemma}
\begin{proof}
Property 2 follows from property 1, which is proven by distinguishing cases of $>_{D}$. By~\cite[Lemma 15]{degroote}, we can only consider the modal rewritings:
\begin{itemize}
\item $(\mathsf{B}_{x} t\,\mathsf{in}\, x) : r = t : \lambda x. r\overline{\overline{x}} = t : \lambda x. r(\lambda h. x h)\overset{\mathsmaller{+}}{>}_{\beta\eta} t : r.$
\item To improve readability, we consider the term $\mathsf{B}_{x}(\mathsf{B}_{y} t\,\mathsf{in}\, s_{1})\,\mathsf{in}\, s$, since the case with multiple variables and subterms is an easy generalization:

$(\mathsf{B}_{x}(\mathsf{B}_{y} t\,\mathsf{in}\, s_{1})\,\mathsf{in}\, s) : r = (\mathsf{B}_{y} t\,\mathsf{in}\, s_{1}) : \lambda x. r\overline{\overline{s}} = t : \lambda y.(\lambda x. r \overline{\overline{s}})\overline{\overline{s_{1}}}>_{\beta} t : \lambda y. r\overline{\overline{s}}[x:=\overline{\overline{s_{1}}}]\underset{\textrm{lem.}\ref{lem13}}{\gg_{\beta\eta}}\\ t : \lambda y. k\overline{\overline{s[x:=s_{1}]}}\,=\,(\mathsf{B}_{y}. t\,\mathsf{in}\,s[x:=s_{1}]) : r$

as required.

\end{itemize}
\end{proof}

Normalization w.r.t. detours is now an easy consequence of the previous lemma:

\begin{proposition}\label{normdet}$\lambda^{\Box\rightarrow\wedge\vee}$ is strongly normalizing w.r.t.~$>_{D}$-reductions.
\end{proposition}
\begin{proof}
By lemma \ref{maind}, every detour elimination from a $\lambda^{\Box\rightarrow\wedge\vee}$-term corresponds to a \mbox{$\beta\eta$-reduction} in simple type theory, which is strongly normalizing~\cite{sorensen}.
\end{proof}

At this point we have to combine the previous results in order to obtain normalization for both permutations and detours. That is, basically, the reason for the next lemma.

\begin{lemma}\label{lemma19} For any $\lambda^{\Box\rightarrow\wedge\vee}$-terms $t,s$ such that $t>_{P}s$, we have the following:
\begin{itemize}
\item $t : r = s : r$ for any term $r$ of simple type theory,
\item $\overline{\overline{t}}=\overline{\overline{s}}$.
\end{itemize}
\end{lemma}
\begin{proof}
Property 2 follows from property 1, which is proven by distinguishing the cases for $>_{P}$. By~\cite[lemma 17]{degroote}, we shall limit to the $\Box,\vee$-permutations, and, as before, we consider the $\Box$-term $\mathsf{B}_{z}(\mathsf{C}_{x,y}(t,s_{1},s_{2}))\,\mathsf{in}\,s$, since the result for the more general structure follows the same pattern:\\
$\mathsf{B}_{z}(\mathsf{C}_{x,y}(t,s_{1},s_{2}))\,\mathsf{in}\,s : r =
\mathsf{C}_{x,y}(t,s_{1},s_{2}) : \lambda z.r\overline{\overline{s}} = \\ = t : \lambda m. m(\lambda x. (s_{1} : \lambda z. r\overline{\overline{s}}))(\lambda y.(s_{2} : \lambda z.r\overline{\overline{s}})) \underset{\mathsmaller{\alpha}}{=}
t : \lambda m. m(\lambda x. (s_{1} : \lambda u. r\overline{\overline{s}}))(\lambda y.(s_{2} : \lambda w.r\overline{\overline{s}})) = \\ = t : \lambda m. m(\lambda x.(\mathsf{B}_{u} s_{1}\,\mathsf{in}\, s : r))(\lambda y. \mathsf{B}_{w} s_{2}\,\mathsf{in}\, s : r)) = \mathsf{C}_{x,y}(t,\mathsf{B}_{u} s_{1}\,\mathsf{in}\, s, \mathsf{B}_{w} s_{2}\,\mathsf{in}\, s) : r$

as desired. 
\end{proof}

Normalization is now a direct consequence.
\begin{theorem}\label{normdetperm}
$\lambda^{\Box\rightarrow\wedge\vee}$ is strongly normalizing w.r.t.~the reduction relation induced by the union of detour- and permutation-conversions.
\end{theorem}
\begin{proof}
Suppose we have an infinite sequence of detour- and permutation-conversions starting with $t_{0}$. We can restrict to three cases:
\begin{itemize}
\item[a.] $t_{0} >_{D} t_{1} \gg_{D} t_{2} \gg_{D}\ldots$ : then by Lemma~\ref{maind}, we have   $\overline{\overline{t_{0}}}\gg_{\beta\eta}\overline{\overline{t_{1}}}\gg_{\beta\eta}\overline{\overline{t_{2}}}\gg_{\beta\eta}\ldots$ in simple type theory, contra normalization for $>_{\beta\eta}$;
\item[b.] $t_{0} \gg_{P} t_{1} \gg_{P} t_{2} \gg_{P}\ldots$: this cannot happen by Lemma~\ref{normperm};
\item[c.] $t_{0} \gg_{D} t_{1} \gg_{P} t_{2} \gg_{D}\ldots$: then by Lemma~\ref{normdet}, in simple type theory $\overline{\overline{t_{2m}}} \gg_{\beta\eta}\overline{\overline{t_{2m+1}}}$, and, by Lemma \ref{lemma19}, $\overline{\overline{t_{2m+1}}}=\overline{\overline{t_{2(m+1)}}}$ for any $m$. Again this is against normalization for $>_{\beta\eta}$.
\end{itemize}

\end{proof}

\subsection{Normalization for $\bot$-conversions}

The previous theorem establishes that the $\bot$-free fragment of $\mathsf{IEL}^{-}$ enjoys the subformula property, along with normalization of proofs.

In order to obtain the same results for the full calculus, we need some further efforts, but the general idea is the same as before.

Let $\lambda^{\Box\rightarrow\wedge\vee\bot}$ denote the type theory corresponding to the full natural deduction for intuitionistic belief.

Its syntax is an extension of that for $\lambda^{\Box\rightarrow\wedge\vee}$ by the type $\bot$ and the term $\mathsf{E}(t : A)$ -- corresponding, as usual, to $\bot$-elimination in the natural deduction.

For permutations, we need to add to Definition~\ref{permut} the following permutation:
 \begin{equation}\label{eqbot}
\mathsf{E}(\mathsf{C}_{x,y}(t,t_{1},t_{2}))>_{P}\mathsf{C}_{x,y}(t,\mathsf{E}(t_{1}),\mathsf{E}({t_{2}})).
\end{equation} 

In \cite{degroote}, de Groote proves that $\mathsf{NJ}$ is strongly normalizing w.r.t.~the reduction relation induced by these extended permutation-conversions and detour-elimination.\footnote{The strategy is again a modified CPS translation of the full calculus and a definition of a norm on untyped terms: see \cite[§ 7]{degroote} for the details.}

To these ones, further $\bot$-conversions are added in order to obtain the subformula property for $\mathsf{NJ}$-normal proofs.
For $\mathsf{IEL}^{-}$, we proceed similarly, and the next definition integrates the additional $\bot$-conversions for $\mathsf{NJ}$ with a specific reduction involving the typing rule for the belief modality.

\begin{definition}[$\bot$-conversions for $\lambda^{\Box\rightarrow\wedge\vee\bot}$]\label{permbot} $\,$\\
\begin{enumerate}
\item $\mathsf{E}(t)s>_{\bot}\mathsf{E}(t)$
\item $\pi_{i}\mathsf{E}(t)>_{\bot}\mathsf{E}(t)$ for $i=1,2$
\item $\mathsf{C}_{x,y}(\mathsf{E}(t),t_{1},t_{2})>_{\bot}\mathsf{E}(t)$
\item $\mathsf{E}(\mathsf{E}(t))>_{\bot}\mathsf{E}(t)$
\item $\mathsf{B}_{x_{1},\cdots,x_{i-1},x_{i},x_{i+1},\cdots,x_{n}}(t_{1},\cdots,t_{i-1},\mathsf{E}(t_{i}),t_{i+1},\cdots,t_{n})\,\mathsf{in}\, s >_{\bot} \mathsf{E}(t_{i}).$
\end{enumerate}
\end{definition}

We immediately have

\begin{proposition}\label{normbot}$\lambda^{\Box\rightarrow\wedge\vee\bot}$ is strongly normalising w.r.t. $\bot$-conversions. 
\end{proposition}
\begin{proof}
Straightforward, for having, in any $\bot$-reduction, a term of smaller complexity on the right hand side of $>_{\bot}$ than that one on the left hand side.
\end{proof}

\section{Analyticity and corollaries}\label{sec3}

$\bot$-conversions are introduced because the permutations of Definition \ref{permut} extended by (\ref{eqbot}) determine the subformula property for $\mathsf{IEL}^{-}$ \emph{only if} $\bot$-elimination only introduces atomic formulas. On the other hand, the reductions of Definition \ref{permbot} -- though necessary for establishing the subformula property of normal $\mathsf{IEL}^{-}$-deductions -- break down the modified CPS translation of $\lambda^{\Box\rightarrow\wedge\vee\bot}$-terms into terms of simple type theory, for normal $\mathsf{IEL}^{-}$-deductions are no longer translated into normal deductions of the implicational fragment of $\mathsf{NJ}$.

Following again \cite{degroote}, in the next sections, we see how to solve the issue, so that we can apply our full computational analysis of intuitionistic belief to establish further properties of the system: the forthcoming lemmas are committed to that.

\subsection{Subformula property}

Gentzen introduced sequent systems in order to give a neat proof of the subformula principle by means of cut-elimination \cite{vonplato}; for both intuitionistic and classical logic, it is possible indeed to translate normal deductions into cut-free derivations in the appropriate sequent calculus and vice-versa, this way obtaining the subformula property for the corresponding natural deduction systems.

In spite of this, it is sometimes possible to give a more direct proof of the principle reasoning about the very natural deduction calculus: here we adopt this strategy.

Recall first from Proposition \ref{normbot} that $\mathsf{IEL}^{-}$ is strongly normalizable w.r.t.~$\bot$-conversions. In order to obtain the subformula property for our calculus we only need to show that we can postpone those conversions, as proven by the following lemmas.

\begin{lemma}\label{lem55} For any $\lambda^{\Box\rightarrow\wedge\vee\bot}$-terms $t,s,r$ such that $t>_{\bot}s$, we have $t[x:=r]>_{\bot}s[x:=r]$.
\end{lemma}
\begin{proof}
Straightforward, after distinguishing the cases for $>_{\bot}$.
\end{proof}

\begin{lemma}\label{lemm56} For any $\lambda^{\Box\rightarrow\wedge\vee\bot}$-terms $t,s,r$ such that $s>_{\bot}r$, we have \\\mbox{$t[x:=s]\gg_{\bot}t[x:=r]$.}
\end{lemma}
\begin{proof}
By induction on $t$.
\end{proof}

\begin{lemma}\label{post} Let $R\in\{D,P\}$ and let $r,s,t$ be $\lambda^{\Box\rightarrow\wedge\vee\bot}$-terms such that $r>_{\bot}s>_{R}t$. Then there exists a $\lambda^{\Box\rightarrow\wedge\vee\bot}$-term $k$ such that $r\overset{\mathsmaller{+}}{>}_{R}k\gg_{\bot}t$.
\end{lemma}
\begin{proof}
By \cite[lemma 57]{degroote}, we only need to consider the following critical pair:
$$\begin{matrix}
\mathsf{B}_{z}(\mathsf{E}(\mathsf{C}_{x,y}(t_{1},t_{2},t_{3})))\,\mathsf{in}\, s_{1} & >_{P} & \mathsf{B}_{z}(\mathsf{C}_{x,y}(t_{1},\mathsf{E}(t_{2}),\mathsf{E}(t_{3})))\,\mathsf{in}\,s_{1}\\
 & >_{P} & \mathsf{C}_{x,y}(t_{1},\mathsf{B}_{z}(\mathsf{E}(t_{2})\,\mathsf{in}\,s_{1}),\mathsf{B}_{z}(\mathsf{E}(t_{3})\,\mathsf{in}\,s_{1})) \\
 & \gg_{\bot} & \mathsf{C}_{x,y}(t_{1},\mathsf{E}(t_{2}),\mathsf{E}(t_{3}))\\
\end{matrix}$$

whenever 

$$\begin{matrix}
\mathsf{B}_{z}(\mathsf{E}(\mathsf{C}_{x,y}(t_{1},t_{2},t_{3})))\,\mathsf{in}\, s_{1} 
 & >_{\bot} & \mathsf{E}(\mathsf{C}_{x,y}(t_{1},t_{2},t_{3})) \\
 & >_{P} & \mathsf{C}_{x,y}(t_{1},\mathsf{E}(t_{2}),\mathsf{E}(t_{3}))\\
\end{matrix}$$

The case for the general modal term is proven similarly.
\end{proof}

At this point, we can finally state our main results.

\begin{theorem}[Full normalization for $\mathsf{IEL}^{-}$]
$\mathsf{IEL}^{-}$ is strongly normalizing w.r.t.~the reduction relation induced by the union of detour-, permutation-, and $\bot$-conversions.
\end{theorem}
\begin{proof}
Suppose otherwise. By repeatedly applying Lemma \ref{post}, we would construct an infinite sequence of $R$-reductions for $R\in\{D,P\}$, contra Theorem \ref{normdetperm}. Therefore, we must have an infinite sequence of $\bot$-conversions, but this is not in the case by Proposition \ref{normbot}.
\end{proof}

\begin{theorem}[Subformula property for $\mathsf{IEL}^{-}$] Every formula occurring in a normal \mbox{$\mathsf{IEL}^{-}$-deduction} of $A$ from assumptions $\Gamma$ is a subformula of $A$ or of some formula in $\Gamma$.
\end{theorem}		
\begin{proof}
We use the main theorem and reason as in \cite[§ II.3]{prawitz1971}.
\end{proof}

\begin{corollary} The calculus $\mathsf{IEL}^{-}$ is decidable.
\end{corollary}
\begin{proof}
By the subformula property, proof search for $\mathsf{IEL}^{-}$ is bounded by the complexity of the formula we wish to deduce in the very calculus.
\end{proof}

\subsection{Further properties}
Having established both normalization and the subformula principle for $\mathsf{IEL}^{-}$, it is now relatively easy to investigate on some further proof-theoretic properties of the system.

First recall that a deduction is said to be \textbf{neutral} iff it consists of a simple assumption, or its last rule is an elimination rule of the natural deduction calculus.

We immediately have the following fact.

\begin{proposition}\label{neutr} In any normal and neutral $\mathsf{IEL}^{-}$-deduction of $A$ from $\Gamma$, $\Gamma\not =\varnothing$.
\end{proposition}
\begin{proof}
Straightforward induction on the height of the deduction.
\end{proof}

From this, we can establish a canonicity result for $\mathsf{IEL}^{-}$-deductions.

\begin{lemma}[Canonicity]\label{canonicity} In any normal $\mathsf{IEL}^{-}$-deduction of $A$, the last rule applied is the introduction rule for the main connective of $A$.
\end{lemma}
\begin{proof}
By Proposition \ref{neutr}, since $\Gamma=\varnothing$ here, and the deduction is normal, its last rule cannot be an elimination.
\end{proof}

\begin{lemma}[Consistency] $\mathsf{IEL}^{-}$ is consistent.
\end{lemma}
\begin{proof}
Consistency follows from canonicity and full normalization, since if we had a deduction of $\bot$, we could normalize it and find a canonical proof, but there is no introduction rule for absurdity in the calculus.
\end{proof}

As further corollaries, we can give syntactic proofs of some structural properties of $\mathsf{IEL}^{-}$.

\begin{corollary} Reflection rule is admissible in $\mathsf{IEL}^{-}$: If $\mathsf{IEL}^{-}\vdash\Box A$, then $\mathsf{IEL}^{-}\vdash A$.
\end{corollary}
\begin{proof}
By Lemma~\ref{canonicity}, since $\Box A$ is deducible from no assumption, its last rule in the corresponding natural deduction must be $\Box$-introduction.
\end{proof}

\begin{corollary}[Disjunction property] If $\mathsf{IEL}^{-}\vdash A\vee B$, then $\mathsf{IEL}^{-}\vdash A$ or $\mathsf{IEL}^{-}\vdash B$.
\end{corollary}
\begin{proof}
We reason as for the previous corollary.
\end{proof}

\begin{corollary}[$\Box$-primality]If $\mathsf{IEL}^{-}\vdash \Box A\vee \Box B$, then $\mathsf{IEL}^{-}\vdash A$ or $\mathsf{IEL}^{-}\vdash B$.
\end{corollary}
\begin{proof}
The result follows by the disjunction property and the admissibility of the reflection rule.
\end{proof}

\begin{corollary}[Modal disjunction property] If $\mathsf{IEL}^{-}\vdash\Box(A\vee B)$, then $\mathsf{IEL}^{-}\vdash\Box A$ or $\mathsf{IEL}^{-}\vdash\Box B$.
\end{corollary}
\begin{proof}
If we have a deduction of $\Box(A\vee B)$, then by the reflection rule we have a deduction of $A\vee B$. By the disjunction property we have a deduction of $A$ or a deduction of $B$. In each case, by applying $\Box$-introduction we have the desired result. 
\end{proof}

\subsection{Proof-theoretic semantics}

In \cite{brogi}, a categorical semantics for proofs in $\mathsf{IEL}^{-}$ is given. As stated before, that work was focused on the computational aspects of the formal system for intuitionistic belief, so that the strong normalization theorem was established for the reduction relation induced by detour-conversions with the modal rewritings here recalled in Definition \ref{riscritture}.

Accordingly, that categorical interpretation for our calculus was developed for \mbox{$\mathsf{IEL}^{-}$-deductions} modulo normalization of the restricted reduction relation.

The structures we were interested in are then the following:\footnote{The interested reader in filling all the defining concepts involved in the definition is referred to e.g.~\cite{maclane}.}

\begin{definition} An $\mathsf{IEL}^{-}$-category is given by a bi-Cartesian closed category $\mathcal{C}$ together with a monoidal pointed endofunctor $\mathfrak{K}$ whose point $\kappa$ is monoidal.
\end{definition}

Indeed, it is proven that these structures capture our deductions in a sound and complete way \cite[§ 3.2]{brogi}.

Considering the normalization result shown in the present paper, we see that, for proving soundness of this interpretation, nothing has to be tweaked from the categorical perspective: the permutation- and $\bot$-conversions that we have considered here are clearly captured by the universal property of colimits.\footnote{This is another evidence that permutations for $\vee$ and $\bot$ do not have computational relevance, as we could have already inferred from the previous results for the CPS translation.}
Therefore we can state the following adequacy theorem.

\begin{theorem}[Categorical soundness] Given an $\mathsf{IEL}^{-}$-category $\mathcal{C}$, any two $\mathsf{IEL}^{-}$-deductions which are equal modulo full normalization are canonically interpreted as $\mathcal{C}$-arrows which are equal. 
\end{theorem}
\begin{proof}
By what we have just remarked about colimits, we can reason just as in \cite[§ 3.2]{brogi}.
\end{proof}

For completeness, the situation is more subtle. It is known that the universal property of colimits imposes stronger rewritings than the usual conversions for $\bot$ and $\vee$ \cite{lambekscott}. This is clearly a shortcoming of the general view-point of categorical semantics for derivations in natural deduction systems, and does not depend on the behaviour of the modality we are now considering: \emph{even for the very \textsf{NJ}}, standard rewritings for absurdity and disjunction are not enough for capturing coproducts. At the same time, the equations imposed by the universal property of coproducts do not look appealing from a proof-theoretic perspective, but cover cases (4) of Definition~\ref{permut} and (5) of Definition~\ref{permbot}.


\section*{Conclusion and related works}
In the present paper we have investigated a deduction system for intuitionistic belief which satisfies many good proof-theoretic properties.

In particular, we have seen that deductions strongly normalise, and that proofs in normal forms satisfy the subformula principle. From these results, we have syntactically proven that the logic is consistent and decidable, along with some structural properties of the system -- namely,  canonicity; disjunction property; admissibility of the reflection rule; $\Box$-primality; modal disjunction property.

As a consequence, we can say to have succeeded in developing a `proof-theoretically tractable' system  for intuitionistic belief that can be easily turned into a modal $\lambda$-calculus, and that is analytic by design.

Moreover, considering normal deductions w.r.t.~both detour elimination and permutations has shown not to invalidate the categorical interpretation introduced in \cite{brogi}.

A natural extension of this investigation goes towards intuitionistic \emph{knowledge} as presented in \cite{artemov}.
As a matter of fact, having a natural deduction for belief, it is not hard to extend our system with an elimination rule corresponding to the principle of intuitionistic factivity of knowledge.

For this extended logic, a cut-free sequent calculus is introduced in \cite{ielseq}. In that perspective, it might be stimulating to develop a sequent calculus on the basis of $\mathsf{IEL}^{-}$ -- following the strategy of \cite{structural} for intuitionistic logic -- and checking the interplay between the Gentzen's formalisms for both intuitionistic belief \emph{and} knowledge.
In fact, it is clearly possible to give a labelled sequent calculus for intuitionistic epistemic states following the pioneering work of \cite{negri}, and obtaining interesting results on the structural behaviour of these modalities by using the powerful insight of formalizing the relational semantics introduced in \cite{artemov}. However, what we have surveyed in the present paper suggests that sticking to the purely syntactic formalism of Gentzen's systems is still a practicable path when dealing with this kind of logics.

The results contained in \cite{ass} -- where a natural deduction for intuitionistic belief is discussed along with algebraic and categorical semantics -- are closer to our viewpoint. Actually, the present author gave talks\footnote{Preliminary results were discussed in May 2019 -- during the Logic and Philosophy of Science seminar at the University of Florence. They were then refined during the poster session of The Proof Society Summer School at Swansea University in September 2019, where valuable feedback from the participants made the author opt for a single ruled calculus.} on the normalization and categorical interpretation for the same calculus discussed there, but soon realised that having a system with a single rule for the belief modality, though lacking a certain symmetry, gives a kind of `kernel' that can be easily extended in several directions -- i.e.~by using various sorts of elimination rules -- enlarging the perspective on the co-reflection scheme beyond the epistemic reading. As expected, having a single rule as ours, or adding, as in \cite{ass}, an elimination rule to the system for $\mathbb{IK}$ defined in \cite{depaivaeike} does not affect the categorical interpretation of the modality.

A fine grained analysis of the pros and cons of those different formal calculi for the intuitionistic epistemic states remains then among future work.

\bibliography{natded}

\begin{thebibliography}{10}

\bibitem{artemov}
Sergei Artemov and Tudor Protopopescu.
\newblock Intuitionistic epistemic logic.
\newblock {\em The Review of Symbolic Logic}, 9.2:266--298, 2016.

\bibitem{degroote}
Philippe de~Groote.
\newblock On the strong normalisation of intuitionistic natural deduction with
  permutation-conversions.
\newblock {\em Information and Computation}, 178.2:441--464, 2002.

\bibitem{depaivaeike}
Valeria de~Paiva and Eike Ritter.
\newblock Basic constructive modality.
\newblock {\em Logic without Frontiers: Festschrift for Walter Alexandre
  Carnielli on the occasion of his 60th Birthday}, pages 411--428, 2011.

\bibitem{kakutani}
Yoshihiko Kakutani.
\newblock Calculi for intuitionistic normal modal logic.
\newblock {\em arXiv preprint arXiv:1606.03180}, 2016.

\bibitem{kripke}
Saul~A. Kripke.
\newblock Semantical analysis of intuitionistic logic i.
\newblock In {\em Studies in Logic and the Foundations of Mathematics},
  volume~40, pages 92--130. Elsevier, 1965.

\bibitem{ielseq}
Vladimir~N. Krupski and Alexey Yatmanov.
\newblock Sequent calculus for intuitionistic epistemic logic iel.
\newblock In {\em International Symposium on Logical Foundations of Computer
  Science}, pages 187--201. Springer, 2016.

\bibitem{lambekscott}
Joachim Lambek and Philip~J. Scott.
\newblock {\em Introduction to higher-order categorical logic}, volume~7.
\newblock Cambridge University Press, 1988.

\bibitem{litak}
Tadeusz Litak, Miriam Polzer, and Ulrich Rabenstein.
\newblock Negative translations and normal modality.
\newblock In {\em 2nd international conference on formal structures for
  computation and deduction (FSCD 2017)}. Schloss Dagstuhl-Leibniz-Zentrum fuer
  Informatik, 2017.

\bibitem{maclane}
Saunders Mac~Lane.
\newblock {\em Categories for the working mathematician}, volume~5.
\newblock Springer Science \& Business Media, 2013.

\bibitem{applicative}
Conor McBride and R.A. Paterson.
\newblock Applicative programming with effects.
\newblock {\em Journal of functional programming}, 18(1):1--13, 2008.

\bibitem{negri}
Sara Negri.
\newblock Proof analysis in modal logic.
\newblock {\em Journal of Philosophical Logic}, 34(5):507--544, 2005.

\bibitem{structural}
Sara Negri, Jan Von~Plato, and Aarne Ranta.
\newblock {\em Structural proof theory}.
\newblock Cambridge University Press, 2008.

\bibitem{brogi}
Cosimo Perini~Brogi.
\newblock Curry-{H}oward-{L}ambek correspondence for intuitionistic belief.
\newblock {\em Studia Logica}, 2021.
\newblock \href {https://doi.org/https://doi.org/10.1007/s11225-021-09952-3}
  {\path{doi:https://doi.org/10.1007/s11225-021-09952-3}}.

\bibitem{prawitz1971}
Dag Prawitz.
\newblock Ideas and results in proof theory.
\newblock In {\em Studies in Logic and the Foundations of Mathematics},
  volume~63, pages 235--307. Elsevier, 1971.

\bibitem{ass}
Daniel Rogozin.
\newblock Categorical and algebraic aspects of the intuitionistic modal logic
  iel―and its predicate extensions.
\newblock {\em Journal of Logic and Computation}, 31(1):347--374, 2021.

\bibitem{sorensen}
Morten~H. S{\o}rensen and Pawel Urzyczyn.
\newblock {\em Lectures on the Curry-Howard isomorphism}, volume 149 of {\em
  Studies in Logic and the Foundations of Mathematics}.
\newblock Elsevier, 2006.

\bibitem{vandalen}
Dirk van Dalen.
\newblock {\em Logic and Structure}.
\newblock Springer, 4th edition, 2008.

\bibitem{troeslavan}
Dirk van Dalen and Anne Troelstra.
\newblock {\em Constructivism in Mathematics. An Introduction I}, volume 121 of
  {\em Studies in Logic and the Foundations of Mathematics}.
\newblock Elsevier, 1988.

\bibitem{vonplato}
Jan von Plato.
\newblock The development of proof theory.
\newblock In Edward~N. Zalta, editor, {\em The Stanford Encyclopedia of
  Philosophy}. Metaphysics Research Lab, Stanford University, winter 2018
  edition, 2018.

\end{thebibliography}

\end{document}